\documentclass[12pt]{amsart}
\usepackage{ifthen,verbatim}
\usepackage{floatflt,graphicx,graphics}
\usepackage{tikz}
\usepackage{mathrsfs}
\usepackage{amsmath,amsfonts,amssymb,amsthm}
\usepackage{subcaption}

\nonstopmode
\numberwithin{equation}{section}
\theoremstyle{plain}
\newtheorem{theorem}[equation]{Theorem}
\newtheorem{conjecture}[equation]{Conjecture}
\newtheorem{lemma}[equation]{Lemma}
\newtheorem{corollary}[equation]{Corollary}

\newtheorem{proposition}[equation]{Proposition}

\theoremstyle{definition}

\newtheorem{remark}[equation]{Remark}
\newtheorem{nonsec}[equation]{}

\theoremstyle{remark}

\newcommand{\R}{\mathbb{R}}

\newcommand{\B}{\mathbb{B}}

\newcommand{\uhp}{\mathbb{H}}

\newcounter{alphabet}

\newcounter{minutes}
\divide\time by 60
\newcounter{hours}
\multiply\time by 60
\addtocounter{minutes}{-\time}

\begin{document}
\bibliographystyle{amsplain}
\title
{Inequalities for geometric mean distance metric}

\def\thefootnote{}
\footnotetext{
\texttt{\tiny File:~\jobname .tex,
          printed: \number\year-\number\month-\number\day,
          \thehours.\ifnum\theminutes<10{0}\fi\theminutes}
}
\makeatletter\def\thefootnote{\@arabic\c@footnote}\makeatother

\author[O. Rainio]{Oona Rainio}

\begin{abstract}
We study a hyperbolic type metric $h_{G,c}$ introduced by Dovgoshey, Hariri, and Vuorinen. We find the best constant $c>0$, for which this function $h_{G,c}$ is a metric in specific choices of $G$. We give several sharp inequalities between $h_{G,c}$ and other hyperbolic type metrics, and also offer a few results related to ball inclusion.
\end{abstract}
\maketitle

\noindent Oona Rainio$^1$, email: \texttt{ormrai@utu.fi}, ORCID: 0000-0002-7775-7656,\newline
1: University of Turku, FI-20014 Turku, Finland\\
\textbf{Funding.} My research was funded by Magnus Ehrnrooth Foundation.\\
\textbf{Acknowledgements.} I am thankful for the referee for their useful suggestions and corrections.\\
\textbf{Data availability statement.} Not applicable, no new data was generated.\\
\textbf{Conflict of interest statement.} There is no conflict of interest.


\section{Introduction}

Hyperbolic geometry is an interesting area of research, crucial for instance to the study of quasiconformal and quasiregular classes of mappings \cite{hkv,v71,vuo}, the fusion of the $n$-dimensional analytic hyperbolic geometry and the special theory of relativity \cite{u22}, the theory of metric spaces with non-positive curvature \cite{b99}, and embeddings of complex networks to hyperbolic spaces \cite{a22,c16,f19,m17}. The hyperbolic metric itself is used in the research of several different types of mappings because of its invariance properties \cite{s20,v71,z21} but this metric can be defined only in special cases for three- or higher-dimensional domains. However, several hyperbolic type metrics have been introduced to be used as substitutes for the hyperbolic metric in general domains \cite{d16,fmv1,GO79,h,inm,s23}.

In a domain $G\subsetneq\R^n$ and for a constant $c>0$, we define a function $h_{G,c}:G\times G\to[0,\infty)$, 
\begin{align*}
h_{G,c}(x,y)=\log\left(1+\frac{c|x-y|}{\sqrt{d_G(x)d_G(y)}}\right),    
\end{align*}
where $d_G(x)$ denotes the Euclidean distance $\inf_{z\in\partial G}|x-z|$ from a point $x\in G$ to the boundary $\partial G$ of $G$ and log is the natural logarithm. This function was briefly presented as a bound for the hyperbolic metric in 1988 by Vuorinen \cite[Proof of Lemma 2.41(2), p. 30]{vuo} and studied in one-dimensional case in 2002 by Hästö \cite[Lemma 5.1, p. 51]{h} but, only in 2016, Dovgoshey, Hariri, and Vuorinen \cite{d16} proved it is a metric for all domains $G\subsetneq\R^n$ if $c\geq2$. More recently, this function $h_{G,c}$ has been researched in \cite{fmv1,s23,w23}. The function $h_{G,c}$ does not have an established name but we suggest the name \emph{geometric mean distance metric} or \emph{geometric mean distance function} to differentiate it from other similar hyperbolic type metrics.

In this article, we continue the earlier research of the function $h_{G,c}$. We both improve and extend the inequalities found earlier for this function. One of our main results is the following theorem concerning the function $h_{G,c}$ in the upper half-space $\uhp^n$. 

\begin{theorem}\label{thm_hHc}
The function $h_{\uhp^n,c}$ is a metric if and only if $c\geq1$. \end{theorem}

The structure of this paper is as follows. In Section 3, we study what is the best constant $c$ for which $h_{G,c}$ fulfills the triangle inequality and give two proofs for Theorem \ref{thm_hHc}. In Section 4, we study the inequalities between $h_{G,c}$ and other hyperbolic type metrics, such as the distance ratio metric and the triangular ratio metric. In Section 5, we study the inclusion properties of the balls defined in the function $h_{G,c}$ between the Euclidean and hyperbolic balls.

\section{Preliminaries}

For a point $x\in G$, let $d_G(x)=\inf_{z\in\partial G}|x-z|$ as in Introduction. For two distinct points $x,y\in\R^n$, let $[x,y]$ be the Euclidean line segment between them. Let $e_1,...,e_n$ be the standard orthonormal basis for $\R^n$ and $x_1,...,x_n$ coordinates of a point $x\in\R^n$ so that $x=(x_1,...,x_n)=\sum^n_{i=1}x_ie_i$. Define now the upper-half space $\uhp^n=\{x\in\R^n\,:\,x_n>0\}$. For $x\in\R^n$ and $r>0$, let $B^n(x,r)$ be the $x$-centered Euclidean open ball with radius $r$ and $S^{n-1}(x,r)$ its boundary sphere. Use the simplified notation $\B^n$ for the unit ball $B^n(0,1)$. Similarly, for $G\subsetneq\R^n$, $x\in G$, $r>0$, and $c>0$, let $B_h(x,r)$ be the $x$-centered open ball with radius $r$ in the semimetric $h_{G,c}$ and $S_h(x,r)$ the corresponding sphere.

Denote the hyperbolic functions by sh, ch, and th, and their inverses by arsh, arch, and arth. The hyperbolic metric has formulas \cite[(4.8), p. 52 \& (4.14), p. 55]{hkv}
\begin{align*}
\text{ch}\rho_{\uhp^n}(x,y)&=1+\frac{|x-y|^2}{2x_ny_n},\quad x,y\in\uhp^n,\\
\text{sh}^2\frac{\rho_{\B^n}(x,y)}{2}&=\frac{|x-y|^2}{(1-|x|^2)(1-|y|^2)},\quad x,y\in\B^n.
\end{align*}
We can write that
\begin{align}\label{for_hypH}
\rho_{\uhp^n}(x,y)
=2\log\left(\frac{|x-y|}{2\sqrt{x_ny_n}}+\sqrt{\frac{|x-y|^2}{4x_ny_n}+1}\right)
=2\text{arsh}\left(\frac{|x-y|}{2\sqrt{x_ny_n}}\right).    
\end{align}
For a domain $G\subsetneq\R^n$, a point $x\in G$, and $r>0$, let $B_\rho(x,r)$ be the $x$-centered open ball with radius $r$ in the hyperbolic metric $\rho_G$ and $S_\rho(x,r)$ the corresponding sphere.

In \cite[Prop. 2.5(1), p. 1466]{d16}, one identity between the function $h_{\uhp^n,c}$ and the hyperbolic metric $\rho_{\uhp^n}$ was already presented but we can obtain a more simplified formula by using \eqref{for_hypH}.

\begin{proposition}\label{prop_idh}
For all $c\geq1$ and $x,y\in\uhp^n$, we have the identity
\begin{align*}
h_{\uhp^n,c}(x,y)=\log\left(1+2c\,{\rm sh}\frac{\rho_{\uhp^n}(x,y)}{2}\right).    
\end{align*}
\end{proposition}
\begin{proof}
Follows from the formula \eqref{for_hypH}.   
\end{proof}

\begin{remark}
It follows from Proposition \ref{prop_idh} and the conformal invariance of the hyperbolic metric that, for any conformal mapping $f:\uhp^n\to\uhp^n=f(\uhp^n)$ and all $x,y\in\uhp^n$, we have
\begin{align*}
h_{\uhp^n,c}(f(x),f(y))=h_{\uhp^n,c}(x,y).    
\end{align*}    
\end{remark}

We will also use the following hyperbolic type metrics for a domain $G\subsetneq\R^n$:
The distance ratio metric introduced by Gehring and Osgood \cite{GO79} $j_G:G\times G\to[0,\infty)$, \cite[p. 685]{chkv}
\begin{align*}
j_G(x,y)=\log\left(1+\frac{|x-y|}{\min\{d_G(x),d_G(y)\}}\right), \end{align*}
its modification called the $j^*$-metric $j^*_G:G\times G\to[0,1],$ \cite[2.2, p. 1123 \& Lemma 2.1, p. 1124]{hvz}
\begin{align*}
j^*_G(x,y)={\rm th}\frac{j_G(x,y)}{2}=\frac{|x-y|}{|x-y|+2\min\{d_G(x),d_G(y)\}},    
\end{align*}
the triangular ratio metric introduced by P. H\"ast\"o in 2002 \cite{h} $s_G:G\times G\to[0,1],$ \cite[(1.1), p. 683]{chkv} 
\begin{align*}
s_G(x,y)=\frac{|x-y|}{\inf_{z\in\partial G}(|x-z|+|z-y|)}, 
\end{align*}
and the point pair function \cite[2.4, p. 1124]{hvz}, $p_G:G\times G\to[0,1),$ \cite[p. 685]{chkv}
\begin{align*}
p_G(x,y)=\frac{|x-y|}{\sqrt{|x-y|^2+4d_G(x)d_G(y)}}.  
\end{align*}
The point pair function is not a metric in all domains $G\subsetneq\R^n$ as it does not always fulfill the triangle inequality. The triangular ratio metric has been recently studied in \cite{sch,sinb,sqm} and the point pair function in \cite{dnrv,spf,fss}.

\section{The constant $c$}

\begin{lemma}\label{lem_c1g}
The function $h_{G,c}$ is not a metric in any domain for $c<1$.  \end{lemma}
\begin{proof}
The function $h_{G,c}$ clearly cannot be a metric for $c\leq0$ because it would not satisfy $h_{G,c}(x,y)>0$ for $x\neq y$. Assume then that $0<c<1$. For $x\in G$, fix $q\in S^{n-1}(x,d_G(x))\cap\partial G$, $z=x+(1-k)(q-x)$, and $y=x+(1-k^2)(q-x)$ with $0<k<1$. We have
\begin{align*}
\lim_{k\to0^+}(h_{G,c}(x,z)+h_{G,c}(z,y)-h_{G,c}(x,y)) 
=\lim_{k\to0^+}\log\left(\frac{(\sqrt{k}+c(1-k))^2}{k+c(1-k^2)}\right)
=\log(c)<0,
\end{align*}
so $h_{G,c}$ does not satisfy the triangle inequality.
\end{proof}

As mentioned already in \cite{d16}, it follows from Hästö's work \cite{h} that $h_{G,c}$ is a metric with $c\geq1$ in $G=\R^n\setminus\{0\}$, but let us explain how this result can be obtained in more detail.

\begin{lemma}\label{lem_c1R}
The function $h_{\R^n\setminus\{0\},c}$ is a metric if and only if $c\geq1$.    
\end{lemma}
\begin{proof}
The function $h_{G,c}$ trivially fulfills the first properties of a metric for all $c>0$, so we only need to consider the triangle inequality. It is clear from the equivalence
\begin{equation}\label{equ_hcG}
\begin{aligned}
&h_{G,c}(x,y)\leq h_{G,c}(x,z)+h_{G,c}(z,y)\\
&\Leftrightarrow\quad
|x-z|\sqrt{d_G(z)d_G(y)}+|y-z|\sqrt{d_G(x)d_G(z)}+c|x-z||z-y|\\
&\quad\quad\,-|x-y|d_G(z)\geq0
\end{aligned}    
\end{equation}
that if $h_{G,c}$ is a metric for $c=1$, then this function is a metric for all $c\geq1$. The function $h_{\R^n\setminus\{0\},1}$ is a metric by \cite[Lemma 5.1, p. 51]{h}. By Lemma \ref{lem_c1g}, $h_{\R^n\setminus\{0\},c}$ is not a metric for $c<1$.
\end{proof}

\begin{remark}
if $(X,d)$ is a Ptolemaic metric space and $p\in X$, then the function
\begin{align*}
\tau_p(x,y)=\log\left(1+\frac{d(x,y)}{\sqrt{d(x,p)d(y,p)}}\right)   
\end{align*}
is a metric on $X\setminus\{p\}$. This result was proved by Aksoy, Ibragimov, and Whiting in \cite{a18} and it can be considered a generalization of Hästö's result \cite[Lemma 5.1, p. 51]{h} because the Euclidean distance is Ptolemaic.
\end{remark}

\begin{nonsec}
\textbf{Proof of Theorem \ref{thm_hHc}.}
\end{nonsec}
\begin{proof}
We only need to prove the triangle inequality $h_{\uhp^n,c}(x,y)\leq h_{\uhp^n,c}(x,z)+h_{\uhp^n,c}(z,y)$ for all $x,y,z\in\uhp^n$. Since $h_{\uhp^n,c}$ is invariant under stretching by a factor $r>0$ and translation in a direction of any vector in $\partial\uhp^n$, we can fix $z=e_n$ without loss of generality. Denote $u=|x-z|\geq|1-x_n|$ and $v=|z-y|\geq|1-y_n|$. Due to the equivalence \eqref{equ_hcG}, the triangle inequality holds if and only if
\begin{align*}
\sqrt{y_n}u+\sqrt{x_n}v+cuv-|x-y|\geq0.   
\end{align*}
Since
\begin{align*}
|x-y|\leq\sqrt{\sum^{n-1}_{i=1}x_i^2+\sum^{n-1}_{i=1}y_i^2+(x_n-y_n)^2}=\sqrt{u^2+v^2-2(1-x_n)(1-y_n)},    
\end{align*}
it is enough to prove that
\begin{align*}
f(u,v)\equiv\sqrt{y_n}u+\sqrt{x_n}v+cuv-\sqrt{u^2+v^2-2(1-x_n)(1-y_n)}\geq0.
\end{align*}
By differentiation,
\begin{align*}
&\frac{\partial}{\partial u}f(u,v)
=\sqrt{y_n}+cv-\frac{u}{\sqrt{u^2+v^2-2(1-x_n)(1-y_n)}}\geq0\\
&\Leftrightarrow\quad
(1-(\sqrt{y_n}+cv)^2)u^2\leq(\sqrt{y_n}+cv)^2(v^2-2(1-x_n)(1-y_n)),
\end{align*}
so, for all $u>0$, $f(u,v)$ can only have a stationary point that is a maximum. Same holds for $v$. Since $\lim_{u\to\infty}f(u,v)=\infty$, the minimum of $f(u,v)$ is obtained when $u=|1-x_n|$ and $v=|1-x_n|$. We can write
\begin{align*}
f_{\min}\equiv f(|1-x_n|,|1-y_n|)=\sqrt{y_n}|1-x_n|+\sqrt{x_n}|1-y_n|+c|1-x_n||1-y_n|-|x_n-y_n|.
\end{align*}
Suppose by symmetry that $x_n\leq y_n$. If $y_n\leq1$,
\begin{align*}
f_{\min}=(\sqrt{y_n}+x_n)(1-\sqrt{y_n})+\sqrt{x_n}(1-y_n)+c(1-x_n)(1-y_n)\geq0.    
\end{align*}
If $1<x_n\leq y_n$,
\begin{align*}
f_{\min}=(\sqrt{y_n}-1)(x_n-1)+(\sqrt{x_n}-1)(y_n-1)+c(x_n-1)(y_n-1)\geq0
\end{align*}
If $x_n\leq1<y_n$ and $c\geq1$,
\begin{align*}
f_{\min}=(\sqrt{y_n}-1)(1-\sqrt{x_n})(c(1+\sqrt{x_n})(1+\sqrt{y_n})-\sqrt{y_n}+\sqrt{x_n})\geq0.
\end{align*}
Consequently, $h_{\uhp^n,c}$ fulfills the triangle inequality if $c\geq1$ and, by Lemma \ref{lem_c1g}, we know that $c=1$ is the best constant here.
\end{proof}

We can also write another proof for Theorem \ref{thm_hHc} by using the identity of Proposition \ref{prop_idh} between the function $h_{G,c}$ and the hyperbolic metric in $\uhp^n$.

\begin{nonsec}
\textbf{Alternate proof of Theorem \ref{thm_hHc}.}
\end{nonsec}
\begin{proof}
If $f:[0,\infty)\to[0,\infty)$ is increasing such that $f^{-1}(0)=\{0\}$ and $f(x)/x$ is decreasing on $(0,\infty)$, then $f$ is a subadditive function and, consequently, $f\circ d$ is a metric for every metric space $(X,d)$. (See, for example, \cite[Thm 1 \& Prop. 2, p. 9]{d98} or \cite[Ex. 5.24(1)-(2), p. 80]{hkv}.) The hyperbolic metric is a metric in $\uhp^n$ and, by Proposition \ref{prop_idh}, $h_{\uhp^n,c}=f\circ\rho_{\uhp^n}$ where $f(x)=\log(1+2c\,{\rm sh}(x/2))$. Clearly, this function $f$ is increasing on $[0,\infty)$ with $f(0)=0$ and, according to Lemma \cite[Lemma 3.4, p. 6]{fmv1}, $f(x)/x$ is decreasing on $(0,\infty)$ for $c\geq1$. Thus, it follows that $h_{\uhp^n,c}$ is a metric if $c\geq1$ and, by Lemma \ref{lem_c1g}, $h_{\uhp^n,c}$ is not a metric if $c<1$.     
\end{proof}

Numerical tests suggest that the following result holds.

\begin{conjecture}\label{conj_hruv}
For any distinct points $u,v\in\R^n$, the function $h_{\R^n\setminus[u,v],c}$ is a metric if and only if $c\geq1$.
\end{conjecture}

If Conjecture \ref{conj_hruv} holds, the next result follows from it.

\begin{conjecture}
For any domain $G\subsetneq\R^n$ such that $\R^n\setminus G$ is a convex, connected set, the function $h_{\uhp^n,c}$ is a metric if and only if $c\geq1$.    
\end{conjecture}

However, it follows from Lemma \ref{lem_guv} below that the constant $c=2$ is the best constant, for instance, in the case of the unit ball, an $n$-dimensional interval, or a twice-punctured real space.

\begin{lemma}\label{lem_guv}
For a domain $G\subsetneq\R^n$ such that $B^n((u+v)/2,|u-v|/2)\subseteq G$ for some $u,v\in\partial G$, the function $h_{G,c}$ is a metric if and only if $c\geq2$.    
\end{lemma}
\begin{proof}
By \cite[Thm 1.1, pp. 1464-1465]{d16}, we know that $h_{G,c}$ is a metric if $c\geq2$ and, by Lemma \ref{lem_c1g}, $h_{G,c}$ is not a metric for $c<1$. Suppose then $1\leq c<2$. If $x=u+k(v-u)$, $y=v+k(u-v)$, and $z=(u+v)/2$ with $0<k<1/2$, we have
\begin{align*}
&\lim_{k\to0^+}(h_{\B^n,c}(x,z)+h_{\B^n,c}(z,y)-h_{\B^n,c}(x,y)) 
=\lim_{k\to0^+}\log\left(\frac{(\sqrt{2k}+c(1-2k))^2}{2(k+c(1-2k))}\right)
=\log\left(\frac{c}{2}\right)\\
&<0,
\end{align*}
so $h_{\B^n,c}$ does not satisfy the triangle inequality.
\end{proof}

We use the function ${\rm th}(h_{G,c}(x,y)/2)$ in Lemmas \ref{lem_jhi1} and \ref{lem_phi} and Corollary \ref{cor_hsi}, so let us present the following result related to it.

\begin{lemma}
For a domain $G\subsetneq\R^n$ and all such values of $c$ for which $h_{G,c}(x,y)$ is a metric, the function
\begin{align*}
{\rm th}\frac{h_{G,c}(x,y)}{2}=\frac{|x-y|}{|x-y|+(2/c)\sqrt{d_G(x)d_G(y)}},\quad x,y\in G,    
\end{align*}
is a metric.
\end{lemma}
\begin{proof}
The result follows from \cite[Thm 1 \& Prop. 2, p. 9]{d98} because ${\rm th}(u/2)$ is increasing on $u\in[0,\infty)$ with ${\rm th}(0/2)=0$ and ${\rm th}(u/2)/u$ is decreasing on $u\in(0,\infty)$.    
\end{proof}

\section{Inequalities}

\begin{lemma}\label{lem_c01ine}
Let $G\subsetneq\R^n$ and $c_1\geq c_0>0$. Then the inequality
\begin{align*}
\frac{c_0}{c_1}h_{G,c_1}(x,y)\leq h_{G,c_0}(x,y)\leq h_{G,c_1}(x,y),    
\end{align*}
holds for all $x,y\in G$ with the best possible constants in the sense that
\begin{align*}
\inf\left\{\frac{h_{G,c_0}(x,y)}{h_{G,c_1}(x,y)}\text{ : }x,y\in G,\,x\neq y\right\}&=\frac{c_0}{c_1}\quad\text{and}\\
\sup\left\{\frac{h_{G,c_0}(x,y)}{h_{G,c_1}(x,y)}\text{ : }x,y\in G,\,x\neq y\right\}&=1.
\end{align*}
\end{lemma}
\begin{proof}
By differentiation,
\begin{align*}
\frac{\partial}{\partial u}\left(\frac{\dfrac{\partial}{\partial u}\log(1+c_0u)}{\dfrac{\partial}{\partial u}\log(1+c_1u)}\right)=\frac{\partial}{\partial u}\left(\frac{c_0(1+c_1u)}{c_1(1+c_0u)}\right)=\frac{c_0(c_1-c_0)}{c_1(1+c_0u)^2}\geq0,   
\end{align*} 
and $\log(1+c_0\cdot0)=\log(1+c_1\cdot0)=0$ so it follows by the monotone form of l'H\^opital's Rule \cite[Thm B.2, p. 465]{hkv} that $\log(1+c_0u)/\log(1+c_1u)$ is increasing on $(0,\infty)$. Its infimum is therefore $\lim_{u\to0^+}(\log(1+c_0u)/\log(1+c_1u))=c_0/c_1$ and its supremum $\lim_{u\to\infty}(\log(1+c_0u)/\log(1+c_1u))=1$. The result follows.
\end{proof}

The inequality of Lemma \ref{lem_c01ine} can be used to study the triangle inequality of the function $h_{G,c}$ defined with $0<c<2$ since we know that $h_{G,2}$ is a metric for all domains $G\subsetneq\R^n$. 

\begin{corollary}
For all $0<c<2$ and any domain $G\subsetneq\R^n$, the function $h_{G,c}$ fulfills the following relaxed version of the triangle inequality for all $x,y,z\in G$:
\begin{align*}
h_{G,c}(x,y)\leq(2/c)(h_{G,c}(x,z)+h_{G,c}(z,y)).   
\end{align*}
\end{corollary}
\begin{proof}
By Lemma \ref{lem_c01ine},
\begin{align*}
h_{G,c}(x,y)\leq h_{G,2}(x,y)\leq h_{G,2}(x,z)+h_{G,2}(z,y)\leq(2/c)(h_{G,c}(x,z)+h_{G,c}(z,y)).     
\end{align*}
\end{proof}

\begin{lemma}\label{lem_jhi}
For $c>0$ and all $x,y\in G\subsetneq\R^n$, the inequality
\begin{align*}
(1/2)\min\{c,1\}j_G(x,y)\leq h_{G,c}(x,y)\leq\max\{c,1\}j_G(x,y) \end{align*}
holds.
\end{lemma}
\begin{proof}
By \cite[Cor. 2.13, p. 1469]{d16} and Lemma \ref{lem_c01ine},
\begin{align*}
(1/2)\min\{c,1\}j_G(x,y)\leq\min\{c,1\}h_{G,1}(x,y)\leq h_{G,c}(x,y).    
\end{align*}
Denote then $d=\min\{d_G(x),d_G(y)\}$ and $u=|x-y|/d$. Since $\max\{d_G(x),d_G(y)\}\geq d$, we have $h_{G,c}(x,y)/j_G(x,y)\leq\log(1+cu)/\log(1+u)$. By applying the monotone form of l'H\^opital's Rule \cite[Thm B.2, p. 465]{hkv} like in the proof of Lemma \ref{lem_c01ine}, we can show the supremum of $\log(1+cu)/\log(1+u)$ is $\max\{c,1\}$.
\end{proof}

\begin{remark}
According to \cite[Lemma 4.4(1), p. 1474]{d16}, for all $c>0$,
\begin{align*}
\frac{c}{2(1+c)}j_G(x,y)\leq h_{G,c}(x,y)\leq cj_G(x,y).  
\end{align*}
However, the upper bound here is incorrect: It does not hold if $0<c<1$ and $d_G(x)=d_G(y)$ because $\log(1+cu)>c\log(1+u)$ for $0<c<1$. Thus, Lemma \ref{lem_jhi} of this paper improves the earlier lower bound of \cite[Lemma 4.4(1), p. 1474]{d16} and offers the correct upper bound.
\end{remark}

\begin{lemma}\label{lem_jhi1}
For all $c>0$ and points $x,y$ in a domain $G\subsetneq\R^n$, the inequality
\begin{align*}
\min\{1,1/c\}{\rm th}\frac{h_{G,c}(x,y)}{2}\leq j^*_G(x,y)\leq\sqrt{1+1/c^2}{\rm th}\frac{h_{G,c}(x,y)}{2} 
\end{align*} 
holds and has the best possible constants in the same sense as in Lemma \ref{lem_c01ine}.
\end{lemma}
\begin{proof}
Denote $d=\min\{d_G(x),d_G(y)\}$. We have
\begin{align*}
\frac{j^*_G(x,y)}{{\rm th}(h_{G,c}(x,y)/2)}
=\frac{|x-y|+(2/c)\sqrt{d_G(x)d_G(y)}}{|x-y|+2d}
\geq\frac{|x-y|+(2/c)\,d}{|x-y|+2d}
\geq\min\{1,1/c\},
\end{align*}
and this lower bound is the limit value of the quotient $j^*_G(x,y)/({\rm th}(h_{G,c}(x,y)/2))$ when $d_G(x)=d_G(y)$ and either $|x-y|\to0^+$ or $|x-y|\to\infty$. It also holds that
\begin{align*}
\frac{j^*_G(x,y)}{{\rm th}(h_{G,c}(x,y)/2)}
\leq\frac{|x-y|+(2/c)\sqrt{d(|x-y|+d)}}{|x-y|+2d}
=\frac{c|x-y|/d+2\sqrt{|x-y|/d+1}}{c(|x-y|/d+2)}.
\end{align*}
By differentiation,
\begin{align*}
&\frac{\partial}{\partial u}\left(\frac{cu+2\sqrt{u+1}}{c(u+2)}\right)=\frac{2c\sqrt{u+1}-u}{c\sqrt{u+1}(u+2)^2}\geq0
\quad\Leftrightarrow\quad
u^2-4c^2u-4c^2\leq0\\
&\Leftrightarrow\quad
u\leq 2c(c+\sqrt{c^2+1}),
\end{align*}
where $u>0$. Consequently, the maximum of $(cu+2\sqrt{u+1})/(c(u+2))$ is obtained at $u=2c(c+\sqrt{c^2+1})$. By inserting this value of $u$, we have
\begin{align*}
\frac{c(2c(c+\sqrt{c^2+1}))+2\sqrt{2c(c+\sqrt{c^2+1})+1}}{c(2c(c+\sqrt{c^2+1})+2)}
=\sqrt{1+1/c^2}.
\end{align*}
The quotient $j^*_G(x,y)/({\rm th}(h_{G,c}(x,y)/2))$ obtains this value if $d_G(y)=d_G(x)+|x-y|$ and $|x-y|=2c(c+\sqrt{c^2+1})d_G(x)$.    
\end{proof}

\begin{lemma}\label{lem_phi}
For all $c>0$ and $x,y$ in any domain $G\subsetneq\R^n$, the inequality
\begin{align*}
\min\{1,1/c\}\,{\rm th}\frac{h_{G,c}(x,y)}{2}\leq p_G(x,y)\leq\sqrt{1+1/c^2}\,{\rm th}\frac{h_{G,c}(x,y)}{2}    
\end{align*}
holds and has the best possible constants in the same sense as in Lemma \ref{lem_c01ine}.
\end{lemma}
\begin{proof}
We can write
\begin{align*}
\frac{p_G(x,y)}{{\rm th}(h_{G,c}(x,y)/2)}
=\frac{|x-y|+(2/c)\sqrt{d_G(x)d_G(y)}}{\sqrt{|x-y|^2+4d_G(x)d_G(y)}}
=\frac{c|x-y|/(2\sqrt{d_G(x)d_G(y)})+1}{c\sqrt{|x-y|^2/(4d_G(x)d_G(y))+1}}.
\end{align*}
By differentiation,
\begin{align*}
\frac{\partial}{\partial u}\left(\frac{cu+1}{c\sqrt{u^2+1}}\right)=\frac{c-u}{c(u^2+1)^{3/2}},    
\end{align*}
so $(cu+1)/(c\sqrt{u^2+1})$ has a maximum $\sqrt{1+1/c^2}$ at $u=c$. The infimum of this quotient is either $\lim_{u\to0^+}(cu+1)/(c\sqrt{u^2+1})=1/c$ or $\lim_{u\to\infty}(cu+1)/(c\sqrt{u^2+1})=1$. 
\end{proof}

\begin{corollary}\label{cor_hsi}
For all $c>0$ and $x,y$ in any domain $G\subsetneq\R^n$, the inequality
\begin{align*}
\min\{1,1/c\}\,{\rm th}\frac{h_{G,c}(x,y)}{2}\leq s_G(x,y)\leq\sqrt{2+2/c^2}\,{\rm th}\frac{h_{G,c}(x,y)}{2}    
\end{align*}
holds and, if $G$ is convex, the constant $\sqrt{2+2/c^2}$ here can be replaced by $\sqrt{1+1/c^2}$.   
\end{corollary}
\begin{proof}
For any domain $G\subsetneq\R^n$, $j^*_G(x,y)\leq s_G(x,y)$ by \cite[Lemma 2.1, p. 1124]{hvz} and $s_G(x,y)\leq\sqrt{2}p_G(x,y)$ by \cite[Thm 3.6, p. 274]{sqm} and, if $G$ is convex, $s_G(x,y)\leq p_G(x,y)$ by \cite[Lemma 11.6(1), p. 197]{hkv}, so the result follows from Lemmas \ref{lem_jhi1} and \ref{lem_phi}.    
\end{proof}

In \cite{d16}, the inequality between the hyperbolic metric $\rho_{\uhp^n}$ and the function $h_{\uhp^n,c}$ was only studied in the case $c\geq2$ but, since we know $h_{\uhp^n,c}$ is a metric for all $c\geq1$, the following result might be useful.

\begin{lemma}
For all $c\geq1$ and $x,y\in\uhp^n$, the inequality
\begin{align*}
(1/c)h_{\uhp^n,c}(x,y)\leq \rho_{\uhp^n}(x,y)\leq 2h_{\uhp^n,c}(x,y)   
\end{align*}
holds and has the best possible constants in the same sense as in Lemma \ref{lem_c01ine}.
\end{lemma}
\begin{proof}
We have $(1/c)h_{\uhp^n,c}(x,y)\leq h_{\uhp^n,1}(x,y)$ for all $c\geq1$ by Lemma \ref{lem_c01ine} and, by writing $u=|x-y|/(2\sqrt{x_ny_n})$ and using the formula \eqref{for_hypH}, we have
\begin{align*}
h_{\uhp^n,1}(x,y)\leq\rho_{\uhp^n}(x,y)
\quad\Leftrightarrow\quad
\log(1+2u)\leq2\log(u+\sqrt{u^2+1})
\quad\Leftrightarrow\quad
u\geq0.
\end{align*}
The proof of the second part is the same as in the proof of \cite[Thm 4.6, p. 1475]{d16}: By \cite[(2.14), p. 23]{vuo}, ${\rm arch}\,t\leq2\log(1+\sqrt{2(t-1)})$, so for all $c\geq1$,
\begin{align*}
\rho_{\uhp^n}(x,y)\leq2\log\left(1+\sqrt{2({\rm ch}(\rho_{\uhp^n}(x,y))-1)}\right)
=2\log\left(1+\frac{|x-y|}{\sqrt{x_ny_n}}\right)\leq 2h_{\uhp^n,c}(x,y).
\end{align*}
Since
\begin{align*}
\lim_{u\to0^+}\frac{2\log(u+\sqrt{u^2+1})}{\log(1+2cu)}=\frac{1}{c},\quad
\lim_{u\to\infty}\frac{2\log(u+\sqrt{u^2+1})}{\log(1+2cu)}=2, 
\end{align*}
we have the best constants here.
\end{proof}

\section{Ball inclusion}

\begin{lemma}\label{lem_huinc}
For any choices of $G\subsetneq\R^n$, $x\in G$, $0<r<d_G(x)$, and $c>0$, the ball inclusion $B_h(x,r_0)\subseteq B^n(x,r)\subseteq B_h(x,r_1)$ holds if $r_0\leq\log\left(1+\dfrac{cr}{\sqrt{d_G(x)(d_G(x)+r)}}\right)$ and $r_1\geq\log\left(1+\dfrac{cr}{\sqrt{d_G(x)(d_G(x)-r)}}\right)$. The lower limit of $r_1$ is the best possible one here for all choices of $G\subsetneq\R^n$.  
\end{lemma}
\begin{proof}
Fix $x\in G$, $0<r<d_G(x)$, $c>0$, and $u\in S^{n-1}(x,d_G(x))\cap\partial G$. For all $y\in S^{n-1}(x,r)$, the value of $h_{G,c}(x,y)$ depends only on $d_G(y)$ since both $d_G(x)$ and $|x-y|=r$ are fixed. The distance $d_G(y)$ obtains its minimum value $d_G(x)-r$ when $y=x+r(u-x)/d_G(x)$, and its maximum has an upper bound $d_G(x)+r$. Thus, the result follows. \end{proof}

The upper limit of $r_0$ in Lemma \ref{lem_huinc} is the best possible one when the domain $G$ is for instance $\R^n\setminus\{0\}$. Namely, in $G=\R^n\setminus\{0\}$, there is only one possible choice of $u\in S^{n-1}(x,d_G(x))\cap\partial G$ for $x\in G$ and the upper bound $d_G(x)+r$ of $d_G(y)$ is obtained when $y=x+r(u-x)/d_G(x)$. However, there is a better inclusion result in the case $G=\B^n$.

\begin{lemma}
For any $x\in\B^n$, $0<r<1-|x|$, and $c>0$, the ball inclusion $B_h(x,r_0)\subseteq B^n(x,r)\subseteq B_h(x,r_1)$ holds if and only if 
\begin{align*}
&r_0\leq\log\left(1+\frac{cr}{\sqrt{(1-|x|)(1-||x|-r|)}}\right)\text{ and }\,
r_1\geq\log\left(1+\frac{cr}{\sqrt{(1-|x|)(1-|x|-r)}}\right).    
\end{align*} 
\end{lemma}
\begin{proof}
The result follows similarly as Lemma \ref{lem_huinc}, but the maximum of $d_G(y)=1-|y|$ is $1-||x|-r|$. 
\end{proof}

\begin{theorem}
For any choices of $x\in G=\uhp^n$, $r>0$, and $c>0$, we have
\begin{align*}
B_h(x,r)
&=B_\rho\left(x,2{\rm arch}\left(\frac{e^r-1}{2c}\right)\right)\\
&=B^n\left(x+\frac{x_ne_n(e^r-1)^2}{2c^2},\frac{x_n(e^r-1)}{2c^2}\sqrt{(e^r-1)^2+4c^2}\right).
\end{align*}
\end{theorem}
\begin{proof}
Fix $x\in\uhp^n$, $r>0$, and $c>0$. By Proposition \ref{prop_idh}, we have $\rho_{\uhp^n}(x,y)=2{\rm arch}((e^r-1)/(2c))$ for all $y\in S_h(x,r)$. By \cite[(4.11), p. 52]{hkv}, $B_\rho(x_ne_n,R)=B^{n-1}(x_ne_n{\rm ch}(R),x_n{\rm sh}R)$ so trivially $B_\rho(x,R)=B^{n-1}(q,x_n{\rm sh}R)$ where $q=x+x_ne_n({\rm ch}(R)-1)$. Simple identities yield, for $u>0$,
\begin{align*}
{\rm ch}(2{\rm arsh}(u))-1&=2{\rm sh}^2({\rm arsh}(u))+1-1=2u^2,\\
{\rm sh}(2{\rm arsh}(u))&=2{\rm sh}({\rm arsh}(u)){\rm ch}({\rm arsh}(u))=2u\sqrt{u^2+1}.
\end{align*}
The result follows.
\end{proof}

\begin{theorem}
For any choices of $x\in G=\B^n$, $R>0$, and $c>0$, the ball inclusion $B_h(x,r_0)\subseteq B_\rho(x,R)\subseteq B_h(x,r_1)$ holds if and only if
\begin{align*}
r_0\leq\log\left(1+\frac{ct(1+|x|)\sqrt{1-|x|}}{(1-|x|t)\sqrt{1-|x|t-||x|-t|}}\right)\quad\text{and}\quad 
r_1\geq\log\left(1+\frac{ct(1+|x|)}{(1+|x|t)(1-t)}\right) 
\end{align*}
with $t={\rm th}(R/2)$.
\end{theorem}
\begin{proof}
Fix $x\in\B^n$, $R>0$, and $c>0$. By \cite[(4.21), p. 56]{hkv}, $B_\rho(x,R)=B^{n-1}(q,r)$ where 
\begin{align*}
q=\frac{x(1-t^2)}{1-|x|^2t^2},\quad
r=\frac{(1-|x|^2)t}{1-|x|^2t^2},\quad\text{and}\quad
t={\rm th}\,\frac{R}{2}.
\end{align*}
Note that $|q-x|=|x|tr$. Let $y\in S_\rho(x,R)$ so that $\mu\in[0,\pi]$ is the angle between the vectors from $q$ to $x$ and $y$. The distance $h_{\B^n,c}(x,y)$ only depends on the quotient
\begin{align*}
\frac{|x-y|}{\sqrt{1-|y|}}
&=\sqrt{\frac{|q-x|^2+|q-y|^2-2|q-x||q-y|\cos(\mu)}{1-\sqrt{|q|^2+|q-y|^2-2|q||q-y|\cos(\pi-\mu)}}}\\
&=\sqrt{\frac{r^2(|x|^2t^2+1-2|x|t\cos(\mu))}{1-\sqrt{|q|^2+r^2+2|q|r\cos(\mu)}}}
\end{align*}
By differentiation,
\begin{align}
&\frac{\partial}{\partial\cos(\mu)}\left(\frac{|x|^2t^2+1-2|x|t\cos(\mu)}{1-\sqrt{|q|^2+r^2+2|q|r\cos(\mu)}}\right)\nonumber\\ 
&=\frac{2|x|t(|q|r\cos(\mu)-\sqrt{|q|^2+r^2+2|q|r\cos(\mu)})+2|x|t(|q|^2+r^2)+|q|r(|x|^2t^2+1)}{\sqrt{|q|^2+r^2+2|q|r\cos(\mu)}(1-\sqrt{|q|^2+r^2+2|q|r\cos(\mu)})^2}\label{diff_cosu}.
\end{align}
Again, by differentiation,
\begin{align*}
&\frac{\partial}{\partial\cos(\mu)}(|q|r\cos(\mu)-\sqrt{|q|^2+r^2+2|q|r\cos(\mu)})=|q|r-\frac{|q|r}{\sqrt{|q|^2+r^2+2|q|r\cos(\mu)}}\geq0\\
&\Leftrightarrow\quad
2|q|r\cos(\mu)\leq1-|q|^2-r^2.
\end{align*}
Consequently, $|q|r\cos(\mu)-\sqrt{|q|^2+r^2+2|q|r\cos(\mu)}$ has the minimum $-(1+|q|^2+r^2)/2$ at $2|q|r\cos(\mu)=1-|q|^2-r^2$. It follows that
\begin{align*}
&2|x|t(|q|r\cos(\mu)-\sqrt{|q|^2+r^2+2|q|r\cos(\mu)})+2|x|t(|q|^2+r^2)+|q|r(|x|^2t^2+1)\\
&\geq-|x|t(1+|q|^2+r^2)+2|x|t(|q|^2+r^2)+|q|r(|x|^2t^2+1)\\
&=|x|t(|q|^2+r^2-1+|q|r|x|t+|q|r/(|x|t))\\
&=\frac{|x|t}{(1-|x|^2t^2)^2}(|x|^2(1-t^2)^2+t^2(1-|x|^2)^2-(1-|x|^2t^2)^2\\
&\quad+(1+t^2|x|^2)(1-t^2)(1-|x|^2))\\
&=\frac{|x|t}{(1-|x|^2t^2)^2}((1+t^2|x|^2)(t^2+|x|^2)-4|x|^2t^2-(1-t^2|x|^2)(t^2+|x|^2)+4|x|^2t^2)\\
&=\frac{2|x|^3t^3}{(1-|x|^2t^2)^2}(t^2+|x|^2)\geq0.
\end{align*}
Thus, the derivative \eqref{diff_cosu} is non-negative and $|x-y|/\sqrt{1-|y|}$ is increasing with respect to $\cos(\mu)$. The minimum of $h_{\B^n,c}(x,y)$ is
\begin{align*}
\log\left(1+c\,\frac{r(1+|x|t)}{\sqrt{(1-|x|)(1-||q|-r|)}}\right)
=\log\left(1+\frac{ct(1+|x|)\sqrt{1-|x|}}{(1-|x|t)\sqrt{1-|x|t-||x|-t|}}\right)
\end{align*}
at $y=q+rx/|x|$ and the maximum of $h_{\B^n,c}(x,y)$ is
\begin{align*}
\log\left(1+c\,\frac{r(1-|x|t)}{\sqrt{(1-|x|)(1-|q|-r)}}\right)
=\log\left(1+\frac{ct(1+|x|)}{(1+|x|t)(1-t)}\right)
\end{align*}
at $y=q-rx/|x|$. The result follows.    
\end{proof}

\end{document}